\documentclass[11pt]{amsart}

\usepackage[utf8]{inputenc}
\usepackage{amsmath, amssymb, amsthm}
\usepackage{mathrsfs}
\usepackage{geometry}
\usepackage{hyperref}
\usepackage{enumitem}
\usepackage{graphicx}
\usepackage{mathtools}
\usepackage{tikz-cd}

\newtheorem{theorem}{Theorem}[section]
\newtheorem{lemma}[theorem]{Lemma}

\newtheorem{corollary}[theorem]{Corollary}

\title{Transverse fibered knots and Taut Foliations}
\author{Rithwik Susheel Vidyarthi}
\address{Department of Mathematics, Michigan State University, East Lansing, MI 48824, USA}
\email{vidyart2@msu.edu}

\begin{document}

\maketitle

\begin{abstract}
We give a short alternative proof of Honda-Kazec-Matic's result, which states that a fibered knot with pseudo-Anosov monodromy and fractional Dehn twist coefficient $\geq1$ supports a contact structure that is a perturbation of a taut foliation.
\end{abstract}

\section{Introduction}
In \cite{HKM2}, Honda-Kazec-Matic proved:
\begin{theorem}[\cite{HKM2} Theorem 1.2]\label{1}
    Let $S$ be a bordered surface with boundary, and $h$ be  pseudo-Anosov with fractional Dehn twist coefficient $c\geq 1$. Then the contact structure carried by the open book decomposition $(S,h)$ is isotopic to a perturbation of a taut foliation.
\end{theorem}
This result has figured prominently throughout the literature, providing a useful criterion for establishing the existence of universally tight semi-fillable contact structures \cite{Bergedual} \cite{Lspace} \cite{QPlinks}  \cite{Lev} \cite{BHS}. For instance, it was used in \cite{Lspace} to prove the L-space conjecture for irreducible three-manifolds containing a genus one fibered knot. In a different direction, it has been shown that the theorem does not hold when you increase the number of boundary components \cite{BE}.
\\ \\
The proof of Theorem \ref{1} in \cite{HKM2} consists of an involved application of convex surface theory, starting from the taut foliations constructed by Roberts in \cite{Rob1}\cite{Rob2}. In this note, we provide a new, much shorter(although indirect) proof of their results relying on the following theorem of \cite{EVM}:

\begin{corollary}[\cite{EVM} Corollary 1.3]\label{EVM}
    If $\xi$ is a tight contact structure on a three manifold $M$, then a fibered knot with pseudo-Anosov monodromy, $L$, has a transverse representative with maximum self-linking number, i.e $sl(K)=2g(K)-1$, in $(M,\xi)$ if and only if $\xi$ is supported by the open book decomposition induced by the fibered knot $L$. Here $g(K)$ denotes the genus of the knot $K$.
\end{corollary}

Given the above, to prove Theorem \ref{1}, it suffices to show that the binding of $(S,h)$ realizes the Bennequin bound in the tight contact structure obtained by perturbing the taut foliation. This requires only a calculation of the self-linking number of the binding in the taut foliations constructed in \cite{Rob1}\cite{Rob2}. We perform this in the following section.
\\ \\
\textbf{Acknowledgments}: I would like to thank my advisor, Matt Hedden, for suggesting this problem to me and for having many fruitful discussions. This project was supported by the grant DMS-2104664 and RTG: Algebraic and Geometric Topology at Michigan State University DMS-2135960.

\section{Self-linking number}
The self-linking number is an invariant of transverse knots in a contact 3-manifold $(M,\xi)$. Let $F$ be a Seifert surface for the knot. Then $TF\cap \xi |_{K}$ gives a non-vanishing section $\sigma$ of $\xi$ along $K$. The self-linking number measures the obstruction to extending this to a non-vanishing section over all of $F$. This is exactly given by the relative Euler class of $\xi$ with respect the section $\sigma$ along $K$, $e_{K}(\xi)\in H^2(M,K)$
We define the self-linking number of $K$ as follows:
$$sl(K)=\langle -e_{K}(\xi), [F]\rangle$$

There is another way to calculate the self-linking number, which we will also use. We can look at the characteristic foliation on $F$ induced by $\xi$. Then, $$sl(K)=-(e_+-h_+)+(e_--h_-)$$
where $e_{\pm}$ are the elliptic singularities and $h_{\pm}$ are the hyperbolic singularities. More details about the self-linking number can be found in the survey \cite{Survey}.

\section{Proof}

Denote the knot exterior by $X=M\backslash \nu(K)$. We have the following inclusions: 
\begin{center}
\begin{tikzcd}
(M,\partial X) \arrow[rd, "i_1"]& & (M,K) \arrow[ld, "i_2"'] \\
& (M,\partial X \cup K)
\end{tikzcd}
\end{center}
We can consider the relative Euler class with respect to a non-vanishing section $\sigma$ on $\partial X \cup K$. This gives non-vanishing sections by restrictions to  $\partial X$ and $K$. Hence, we get the following relative Euler classes: $e\in H^2(M,\partial X\cup K)$, $e_{\partial X}\in H^2(M,\partial X)$ and $e_{K}\in H^2(M,K)$. Conversely, we can start with a section on $\partial X$ and $K$ to obtain a section on $\partial X \cup K$. From the naturality of relative Euler classes, we have:
$$i_1^*(e)=e_{\partial X}$$
$$i_2^*(e)=e_{K}$$

Let $\mathscr{F}$ be the taut foliation on $M$ constructed in \cite{Rob1}\cite{Rob2}. This is obtained by first constructing a taut foliation on $X$ and then extending it to all of $M$ by meridional disks. The condition of the fractional Dehn twist coefficient $c\geq 1$ ensures that we realize the boundary slope corresponding to infinity surgery. Now $\mathscr{F}$ can be perturbed to obtain a contact structure $\xi$ on $M$ by Eliashberg-Thurston in \cite{ET}:
\begin{theorem}[Theorem 2.9.1 \cite{ET}]
   For any $C^0$-foliation $\xi$ on an oriented 3-manifold $M$, one can find a family of smooth confoliations $\xi_t$ which continuously depends on the parameter $t$, such that $\xi_0=\xi$, $\xi_1$ is a contact structure, and $\xi_t$ is $C^0$-close to $\xi$ for all $t\in [0,1]$.
\end{theorem}

\begin{corollary}[Corollary 3.2.5 \cite{ET}]
    Contact structures, $C^0$-close to a taut foliation, are symplectically semi-fillable, and therefore, tight.
\end{corollary}

Now we prove that under these perturbations, the transversality and self-linking number don't change.

\begin{lemma}
    Suppose $\xi$ and $\xi_t$ as above. Let $K$ be a simple closed curve that is transverse to $\xi$. Then $K$ is transverse to $\xi_t$  for all $ t \in [0,1]$. Moreover, the self-linking numbers of $K$ is the same in all $\xi_t$.
\end{lemma}
\begin{proof}
   Assume that $\xi=Ker(\alpha)$ and $\xi_t=\ker(\alpha_t)$ for a continuous family of 1-forms $\alpha_t, 0\leq t\leq 1$ with $\alpha_0=\alpha$. Now the condition that $K$ is transverse to $\xi$ is exactly $\alpha(K)\neq 0$, which is an open condition. Since $\alpha_t$ is $C^0$-close to $\alpha_t$, we have $\alpha(K)\neq0\implies \alpha_t(K)\neq0 $. $K$ is transverse to $\xi_t$.\\
   Since the one-parameter family is continuous, the self-linking number defines a continuous function from $t\in[0,1]$ to $\mathbb{Z}$, so it has to be a locally constant function which proves the second part of the lemma.
\end{proof}
Let $F$ be a Seifert surface for $K$ in $M$. Then we can write $F=F'\cup A$, where $F'$ is a properly embedded surface in $X=M\backslash \nu(K)$ and $A$ is an annulus properly embedded in $\nu(K)$ with $\partial A=\partial F'\cup K$. Let $\sigma_{\partial X}$ be a nowhere vanishing section of $\mathscr{F}|_{\partial X}$ is transverse to $\partial X$ pointing outwards, and $e_{\partial X}\in H^2(M,\partial X)$ the corresponding Euler class. The relative Euler classes $e_{\partial X}$ have already been computed in \cite{Hu}. There it is shown that $\langle e_{\partial X},[F']\rangle=1-2g$. \\ \\
Now we let $e_K\in  H^2(M, K)$ denote the relative Euler class corresponding to the section $\sigma_K$ on $K$ which comes from the Seifert surface $F$. Let $e\in H^2(M,\partial X\cup K)$ be the Euler class corresponding to the section $\sigma_{\partial X}\cup \sigma_K$ of $\mathscr{F}$.

\begin{figure}[t]
\centering
\includegraphics[width=8cm]{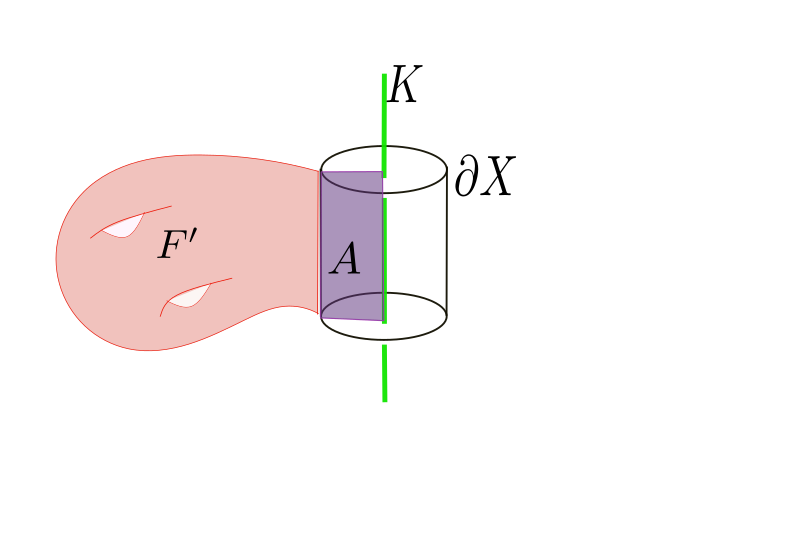}
\caption{The Seifert surface}
\end{figure}

\begin{lemma}
    $\langle  e,[F']\rangle =1-2g$
\end{lemma}
\begin{proof}
    From \cite{Hu}, we know that $\langle e_{\partial X},[F'] \rangle=1-2g$. So:
    \begin{align*}
        &\langle e_{\partial X},[F']\rangle=1-2g\\
        \implies & \langle(i_1)^*(e),[F']\rangle=1-2g\\
        \implies & \langle e, (i_1)_*[F'] \rangle =1-2g\\
        \implies & \langle e,[F']\rangle=1-2g \\
    \end{align*}
\end{proof}

\begin{lemma}\label{lem1}\label{lem1}
    $\langle e_K,[F]\rangle=1-2g \iff \langle e,[A]\rangle=0$
\end{lemma}
\begin{proof}
    \begin{align*}
        &\langle e_K,[F] \rangle=1-2g\\
        \iff & \langle i_2^*(e),[F]\rangle=1-2g\\
        \iff &\langle  e ,(i_2)_*[F]\rangle=1-2g\\
        \iff &\langle  e,(i_2)_*[F'\cup A]\rangle=1-2g\\
        \iff &\langle  e,(i_2)_*([F']+[A])\rangle=1-2g\\
        \iff & \langle e,(i_2)_*[F']\rangle+ \langle e,(i_2)_*[A]\rangle=1-2g\\
        \iff & \langle e,[F']\rangle+\langle e,[A]\rangle=1-2g\\
        \iff &  1-2g +\langle  e,[A]\rangle=1-2g\\
        \iff & \langle e, [A]\rangle=0
    \end{align*}
\end{proof}

Now, in order to calculate $\langle e,[A]\rangle$, notice that since $A$ is a Seifert surface for the link $L=\partial F'\cup K$, this is just the self-linking number of the transverse link $L$. We can compute this by analyzing the characteristic foliation induced by $\mathscr{F}$ on the Seifert surface of $L$, which is the annulus $A$. 

\begin{lemma}\label{lem2}
    $\langle e, [A]\rangle$=0
\end{lemma}
\begin{proof}
    From the above discussion, we just need to analyze the characteristic foliation on $A$ induced by $\mathscr{F}$. Since the foliation is extended by meridional disks, which are transverse to the annulus, the foliation on the annulus is by parallel lines. So, there are no singularities. Hence, the self-linking number of the link $L=\partial F' \cup K$ is zero.
\end{proof}

\begin{proof}[Proof of Theorem \ref{1}]
    We see that the self-linking number of $K$ in $\mathscr{F}$ is $2g-1$ using Lemma \ref{lem1} and Lemma \ref{lem2}. Since this is the same as $sl(K)$ in the contact structure $\xi$, using Corollary \ref{EVM}, we conclude the theorem.
\end{proof}

\bibliographystyle{alpha}
\bibliography{references} 

\end{document}